\documentclass[english, a4paper, 12pt]{article}
\usepackage{amsmath}
\usepackage{amsthm}
\usepackage{amssymb}
\usepackage{amsfonts}
\usepackage{mathtools}
\usepackage{mathrsfs}
\usepackage{bbm}
\usepackage{enumitem}

\newcommand{\e}{\mathbb{E}}

\newcommand{\p}{\mathbb{P}}
\newtheorem{defi}{Definition}[section]
\newtheorem{lemma}[defi]{Lemma}

\newtheorem{theorem}[defi]{Theorem}
\newtheorem{kor}[defi]{Corollary}

\newtheorem{rem}[defi]{Remark}
\newtheorem{ex}[defi]{Example}

\begin{document}

\title{On Lamperti transformation and characterisations of discrete random fields}

\renewcommand{\thefootnote}{\fnsymbol{footnote}}

\author{Marko Voutilainen\footnotemark[1],\, Lauri Viitasaari\footnotemark[2],\, Pauliina Ilmonen\footnotemark[3]}

\footnotetext[1]{Turku School of Economics, Department of Accounting and Finance, FI-20014 University of Turku, Finland}

\footnotetext[2]{Uppsala University, Department of Mathematics, Box 480, 751 06 Uppsala, Sweden}

\footnotetext[3]{Aalto University, Department of Mathematics and Systems Analysis, P.O. Box 11100, FI-00076 Aalto, Finland}

\maketitle

\begin{abstract}
\noindent
In this article we characterise discrete time stationary fields by difference equations involving stationary increment fields and self-similar fields. This gives connections between stationary fields, stationary increment fields and, through Lamperti transformation, self-similar fields. Our contribution is a natural generalisation of recently proved results covering the case of stationary processes.
\end{abstract}

{\small
\medskip

\noindent
\textbf{AMS 2010 Mathematics Subject Classification:} 60G60, 60G10, 60G18
\medskip

\noindent
\textbf{Keywords:} random fields, stationary fields, self-similar fields, Lamperti transformation, fractional Ornstein-Uhlenbeck fields
}


\section{Introduction}
Stationary processes $X = (X_t)_{t\in T}$ have numerous applications in many different fields, and they are a topic of active research. Similarly, self-similar processes and stationary increment processes have many applications in various disciplines of science. For details on self-similar processes, we refer to the monograph \cite{embrechts2002selfsimilar} and the references therein.

All of these three classes are intimately connected. Indeed, it was already observed by Lamperti in \cite{Lamperti} that there exists a one-to-one correspondence between stationary processes and self-similar processes. Later on, this connection was used in \cite{viitasaari2016} to obtain relation between stationary processes and stationary increment processes in continuous time, i.e. $T = \mathbb{R}$, through Langevin equation
\begin{equation}
\label{eq:langevin-intro}
dX_t = - \theta X_t dt + dG_t,
\end{equation}
where $\theta>0$ is a parameter, $X$ is stationary, and $G$ has stationary increments (along with certain other properties). Most notably, this gives rise to the well-known Ornstein-Uhlenbeck process when one plugs in $G = W$, the Brownian motion. Connection \eqref{eq:langevin-intro} was later extended for discrete time processes $X$, i.e. $T = \mathbb{Z}$, in \cite{voutilainen2017model}, where the authors proved discrete analogue 
\begin{equation}
\label{eq:difference-intro}
\Delta X_{t} = - \theta X_{t-1} + \Delta G_t
\end{equation}
of \eqref{eq:langevin-intro}, and studied the estimation of the unknown parameter $\theta$. A vector-valued version was later provided in \cite{voutilainen2021vector} and \cite{voutilainen2020modeling}, covering both continuous and discrete time cases together with estimation of the unknown parameter matrix.

Similarly to stationary processes, stationary fields form an important subclass of random objects. In this case, $X=(X_t)_{t\in T}$ with $T=\mathbb{R}^N$ in continuous time or $T= \mathbb{Z}^N$ in discrete time (naturally, $T$ can be a more general parameter space). Also, self-similarity and stationarity of the increments are wanted features in many applications.
However, while the notion of stationarity is essentially unchanged in the context of random fields, the notion of self-similarity and stationarity of the increments become more complicated when $t$ is multidimensional. For notion of self-similarity for fields, one typically introduces componentwise self-similarity and considers $H$-self-similarity with $H$ as an $N$-dimensional vector of componentwise self-similarity indices, see e.g. \cite{samorodnitsky1996stable, genton2007self}. In \cite{genton2007self} a version of the Lamperti theorem was proved (in continuous time) for fields, providing a connection between stationary fields and self-similar fields. For other notions of self-similarity for fields, see for example \cite{clausel2010gaussian} and \cite{bierme2007operator}.

The notion of stationary increments for fields is even more complicated due to the fact that the definition of increment is not obvious. One approach is to consider \emph{rectangular increments}, where increments are taken over $N$-dimensional rectangulars. Gaussian self-similar fields and Gaussian rectangular increment fields were studied, again in continuous time, e.g. in \cite{makogin2019gaussian} and \cite{makogin2015example}. 

In this article we extend the Characterisation \eqref{eq:difference-intro} provided in \cite{voutilainen2017model} to discrete time fields, providing a connection between stationary fields, self-similar fields, and stationary increment fields. More precisely, we provide a characterisation of the type (see Theorem \ref{theo:giff})
$$
X_t = \langle\hat{\Theta},\hat{X}^-_t\rangle + \Delta_t G.
$$
Here $\hat{\Theta}$ is a vector of parameters, and $\hat{X}_t^-$ is a vector of ''previous value'' consisting of previous values in different coordinate directions. 
Our notion of increment $\Delta_t G$ corresponds to the notion of stationary rectangular increments of \cite{makogin2019gaussian}, cf. Remark \ref{rem:connection}. As such and exactly as in \cite{viitasaari2016,voutilainen2017model} in the case of processes, we obtain correspondence between stationary fields, stationary rectangular increment fields, and self-similar fields.

The rest of the article is organised as follows. In Section \ref{sec:main} we introduce and prove our main results. We introduce our notation and main definitions in Section \ref{subsec:prel}, while our main results and their proofs are presented in Section \ref{subsec:results}. In Section \ref{subsec:example} we briefly illustrate how our characterisation can be used to construct discrete time fractional Ornstein-Uhlenbeck fields, extending notions of (generalized) Ornstein-Uhlenbeck processes of \cite{viitasaari2016}. We end the paper with conclusions, Section \ref{sec:conclusion}, describing future directions, in particular to cover continuous time parameter space and statistical inference.

\section{Connections between stationary, self-similar, and stationary increment fields}
\label{sec:main}
\subsection{Preliminaries and notations}
\label{subsec:prel}
We begin with by introducing some definitions and notations.
\begin{defi}[Stationarity]
A random field $X = (X_t)_{t\in\mathbb{Z}^N}$ is stationary if
$$(X_{t+s})_{t\in\mathbb{Z}^N} \overset{\text{law}}{=} (X_t)_{t\in\mathbb{Z}^N}$$
for every $s\in\mathbb{Z}^N$ in the sense of finite dimensional distributions.
\end{defi}

\begin{defi}[Self-similarity]
\label{defi:self-similarity}
Let $Y = (Y_{e^t})_{t\in\mathbb{Z}^N} = (Y_{e^{t_1}, \dots, e^{t_N}})_{t\in\mathbb{Z}^N}$ be a random field. In addition, let $\Theta = (\theta_1, \dots, \theta_N)\in (0,\infty)^N$ be a positive multi-index.
If 
$$(Y_{e^{t+s}})_{t\in\mathbb{Z}^N} \overset{\text{law}}{=} (e^{\langle s,\Theta\rangle} Y_{e^t})_{t\in\mathbb{Z}^N}$$
for every  $s\in\mathbb{Z}^N$, where $\langle s,\Theta\rangle$ is the standard inner product of vectors, then $Y$ is a $\Theta$-self-similar random field.
\end{defi}
\begin{rem}
The exponential terms in Definition \ref{defi:self-similarity} are introduced in order to take into account the discrete nature of the field. Definition \ref{defi:self-similarity} is analogous to the classical definition in continuous time. See e.g. \cite{Lamperti} for the definition in the one parameter continuous case and \cite{viitasaari2016} for the definition in the one parameter discrete case.
\end{rem}
The following definition provides a notion of Lamperti transformation in our setting.
\begin{defi}[Lamperti]
\label{defi:lamperti}
Let $\Theta = (\theta_1, \dots, \theta_N)\in (0,\infty)^N$. The Lamperti transformation $\mathcal{L}_\Theta$ and its inverse $\mathcal{L}^{-1}_\Theta$ for discrete random fields are defined by
\begin{align*}
(\mathcal{L}_\Theta X)_{e^t} &= e^{\langle t,\Theta\rangle} X_t, \quad t\in\mathbb{Z}^N,\\
(\mathcal{L}^{-1}_\Theta Y)_t &= e^{-\langle t,\Theta\rangle} Y_{e^t}, \quad  t\in\mathbb{Z}^N.
\end{align*}
\end{defi}

\begin{rem}
The formulae in Definitions \ref{defi:self-similarity} and \ref{defi:lamperti} differ slightly from the standard forms in continuous settings. In the case of random fields, the definition of component-wise self-similarity and the corresponding Lamperti transformation together with a one-to-one correspondence between self-similar and stationary fields was presented in \cite{genton2007self}. In comparison, our definitions are obtained via change of variables from the standard ones, ensuring that we stay within our discrete parameter set when applying the transformation.
\end{rem}

\begin{defi}[Increments]
The square increment $\Delta_t X$ of a field $(X_t)_{t\in\mathbb{Z}^N}$ at a point $t = (t_1,\dots,t_N)$ $\in\mathbb{Z}^N$ is given by
\begin{equation}
\label{eq:increments}
\Delta_t X = \sum_{(i_1,\dots,i_N)\in\{0,1\}^N} (-1)^{\sum_{l=1}^N i_l} X_{t_1-i_1,\dots,t_N-i_N}.
\end{equation}
\end{defi}
\begin{ex}
In the particular case $N=2$, we have 
$$
\Delta_t X =X_{t_1,t_2} + X_{t_1-1,t_2-1} - X_{t_1-1,t_2} - X_{t_1,t_2-1}.
$$
\end{ex}
\begin{defi}[Stationary increment field]
\label{def:stat-inc}
A field $X = (X_t)_{t\in\mathbb{Z}^N}$ has stationary increments if the increment field $(\Delta_t X)_{t\in\mathbb{Z}^N}$ is stationary. That is
$$(\Delta_{t+s} X)_{t\in\mathbb{Z}^N} \overset{\text{law}}{=} (\Delta_t X)_{t\in\mathbb{Z}^N}$$
for every $s\in\mathbb{Z}^N$ in the sense of finite dimensional distributions.
\end{defi}

\begin{rem}
\label{rem:connection}
The authors in \cite{makogin2019gaussian} introduced a continuous time analogous notion of \emph{strictly stationary rectangular increments} by assuming stationarity of the increments over arbitrary rectangular increments. In comparison, in our definition, we consider stationary increments over unit rectangulars. However, due to the discrete nature of our index space $t\in \mathbb{Z}^N$, one can show that our definition is equivalent to assuming stationarity over arbitrary (discrete) rectangular increments. 
\end{rem}

We also need a notion of \emph{previous value} that is not so straightforward in a multi-parameter setting.
\begin{defi}[Previous value]
The previous value of the field $X=(X_t)_{t\in\mathbb{Z}^N}$ at a point $t=(t_1,\dots,t_N)\in\mathbb{Z}^N$ is given by
$$X_t^-  = X_t - \Delta_t X = \sum_{\substack{(i_1,\ldots,i_N)\in\{0,1\}^N\\ (i_1,\ldots,i_N)\neq \bf{0}}} (-1)^{1+\sum_{l=1}^N i_l} X_{t_1-i_1,\dots,t_N-i_N}.$$
\end{defi}
\begin{rem}
Note that in our definition of the previous value, we take into account the terms in \eqref{eq:increments} that have a smaller index in at least one of the coordinate directions. Indeed, in the two-dimensional case, we have
$$
\Delta_t X = X_{t_1,t_2} + X_{t_1-1,t_2-1} - X_{t_1-1,t_2} - X_{t_1,t_2-1}
$$
while the previous value is given by
$$
X^-_t = X_{t_1-1,t_2} + X_{t_1,t_2-1} - X_{t_1-1,t_2-1}.
$$
\end{rem}
\begin{defi}[Inner product $\langle \hat{\Theta},\hat{X}^-_t\rangle$]
\label{defi:inner2}
Let $\Theta = (\theta_1, \dots, \theta_N)\in (0,\infty)^N$ and $X=(X_t)_{t\in\mathbb{Z}^N}$ be a field. We define vectors $\hat{\Theta}$ and $\hat{X}^-_t$ of length $2^N-1$ having elements of the forms
$$ (-1)^{1+\sum_{l=1}^N i_l}e^{-\langle i,\Theta\rangle} \quad\text{and}\quad  X_{{t_1-i_1},\dots,{t_N-i_N}},$$
respectively, where $i=(i_1,\dots,i_N)\in\{0,1\}^N, i\neq \bf{0}.$
Then the inner product of the vectors is
$$\langle\hat{\Theta},\hat{X}^-_t\rangle= \sum_{\substack{(i_1,\dots,i_N)\in\{0,1\}^N\\ i\neq \bf{0}}} (-1)^{1+\sum_{l=1}^N i_l} e^{-\langle i,\Theta\rangle} X_{{t_1-i_1},\dots,{t_N-i_N}}.$$
\end{defi}
\begin{ex}
In the two-dimensional case we obtain that, for $\Theta = (\theta_1, \theta_2)\in (0,\infty)^2$ and $X=(X_t)_{t\in\mathbb{Z}^2}$, the vectors $\hat{\Theta}$ and $\hat{X}^-_t$ are given as
$$\hat{\Theta} = (e^{-\theta_1}, e^{-\theta_2}, -e^{-\theta_1 -\theta_2})\quad\text{and}\quad 
\hat{X}_t^- = \begin{pmatrix}
X_{t_1-1, t_2}\\
X_{t_1, t_2-1}\\
X_{t_1-1, t_2-1}
\end{pmatrix}
$$
and we have
\begin{equation}
\label{eq:innerproduct-2d}
\langle\hat{\Theta},\hat{X}^-_t\rangle= e^{-\theta_1} X_{t_1-1, t_2} + e^{-\theta_2} X_{t_1, t_2-1} - e^{-\theta_1 -\theta_2} X_{t_1-1,t_2-1}.
\end{equation}
\end{ex}
\begin{rem}
The vectors $\hat{\Theta}$ and $\hat{X}^-_t$ are unique up to permutations of their elements.
\end{rem}

\begin{defi}[Class $\mathcal{G}_{{\Theta}}$]
\label{defi:class}
Let $\Theta = (\theta_1, \dots, \theta_N)\in (0,\infty)^N$. Let $G = (G_t)_{t\in\mathbb{Z}^N}$ be a stationary increment field with $G_t = 0$ for all $t$ such that $\sum_{i=1}^N t_i \in\{0,-1,\dots,-N+1\}$. If 
\begin{equation}
\label{eq:glimitcondition}
\lim_{M_1\to\infty} \dots \lim_{M_N\to\infty} \sum_{j_1=-M_1}^{t_1}\dots \sum_{j_N=-M_N}^{t_N}  e^{\sum_{l=1}^N j_l\theta_l} \Delta_{(j_1,\dots,j_N)}G
\end{equation}
converges in probability defining an almost surely finite random variable for every $t=(t_1,\dots,t_N)\in\mathbb{Z}^N$, then $G\in\mathcal{G}_{{\Theta}}$.
\end{defi}
\begin{rem}
It turns out that the order of the limits in \eqref{eq:glimitcondition} is irrelevant, and any permutation of the order leads to the convergence towards the same limiting random variable that turns out to be the stationary field $X_t$, cf. proof of Theorem \ref{theo:gchar1}.
\end{rem}
\begin{rem}
Condition $G_t = 0$ for all $t$ such that $\sum_{i=1}^N t_i \in\{0,-1,\dots,-N+1\}$ is rather peculiar and it essentially means that $G$ has to vanish along discrete points of certain $N-1$-dimensional planes. However, this condition is not required \emph{a priori} for the characterisation, but turns out to hold true and is also required to obtain uniqueness of the representation, cf. Theorem \ref{theo:giff} and Remark \ref{rem:peculiar-condition}.
\end{rem}
\begin{rem}
\label{rem:integrability}
Note that if $G\in L^1$ (or $\Delta G\in L^1$), then $G\in \mathcal{G}_{{\Theta}}$ for every $\Theta$. Indeed, this can be seen from
\begin{equation*}
\begin{split}
&\e |\sum_{j_1=-\infty}^{t_1}\dots \sum_{j_N=-\infty}^{t_N}  e^{\sum_{l=1}^N j_l\theta_l} \Delta_{(j_1,\dots,j_N)}G|\\
 &\leq \sum_{j_1=-\infty}^{t_1}\dots \sum_{j_N=-\infty}^{t_N}  e^{\sum_{l=1}^N j_l\theta_l} \e|\Delta_{(j_1,\dots,j_N)}G|\\
 &=\sum_{j_1=-\infty}^{t_1}\dots \sum_{j_N=-\infty}^{t_N}  e^{\sum_{l=1}^N j_l\theta_l} \e|\Delta_{(1,\dots,1)}G|\\
 &\leq \sum_{j_1=-\infty}^{t_1}\dots \sum_{j_N=-\infty}^{t_N}  e^{\sum_{l=1}^N j_l\theta_l} \sum_{(i_1,\dots,i_N)\in\{0,1\}^N}\e|G_{1-i_1,\dots,1-i_N}|.
\end{split}
\end{equation*}
\end{rem}
\begin{ex}
In the two-dimensional case, we have $G\in\mathcal{G}_{{\Theta}}$ for $\Theta = (\theta_1,\theta_2) \in (0,\infty)^2$ provided that $G$ has stationary increments, $G_{t,-t} = G_{t,-t-1}=0$ and 
$$
\lim_{M_1 \to \infty}\lim_{M_2\to\infty} \sum_{j_1=-M_1}^{t_1} \sum_{j_2=-M_2}^{t_2} e^{j_1\theta_1}e^{j_2\theta_2} \Delta_{(j_1,j_2)}G
$$
exists as an almost surely finite random variable. That is, $G$ has stationary increments and $G_{t_1,t_2}$ is set to zero on lines $t_1 = - t_2$ and $t_1 = -t_2-1$. The existence of such fields follow as a by-product of our main results. 
\end{ex}

\subsection{AR(1) type characterisation of stationary fields}
\label{subsec:results}
Our main result is the following characterisation that is a natural extension of the one-dimensional case presented in \cite{voutilainen2017model}.
\begin{theorem}
\label{theo:giff}
Let $\Theta = (\theta_1, \dots, \theta_N)\in (0,\infty)^N$. A field $X=(X_t)_{t\in\mathbb{Z}^N}$ is stationary if and only if the following conditions are satisfied.
\begin{enumerate}[label=(\roman*)]
\item $$\lim_{m\to -\infty} e^{m\theta_j}X_{t_1,\dots,t_{j-1},m,t_{j+1},\dots,t_N} \overset{\p}{\longrightarrow} 0$$
for every $j\in\{1,\dots,N\}$ and $t_1, \dots,t_{j-1},t_{j+1},\dots, t_N \in \mathbb{Z}$.
\item There exists $G=(G_t)_{t\in\mathbb{Z}^N}\in\mathcal{G}_{{\Theta}}$ such that
\begin{equation}
\label{eq:grecursion3}
X_t = \langle\hat{\Theta}, \hat{X}_t^-\rangle + \Delta_t G\quad\text{for every } t\in\mathbb{Z}^N,
\end{equation}
where $\langle \hat{\Theta}, \hat{X}_t^-\rangle$ is given by Definition \ref{defi:inner2}. 
\end{enumerate}
Moreover, for a given $\Theta$, the stationary increment field $G\in\mathcal{G}_{{\Theta}}$ in \eqref{eq:grecursion3} is unique.
\end{theorem}
As a direct corollary we obtain the following version in a two-dimensional case.
\begin{kor}
Let $\Theta = (\theta_1, \theta_2)\in (0,\infty)^2$. A field $X=(X_t)_{t\in\mathbb{Z}^2}$ is stationary if and only if the following conditions are satisfied.
\begin{enumerate}[label=(\roman*)]
\item $$\lim_{m\to -\infty} e^{m\theta_2}X_{t_1,m} \overset{\p}{\longrightarrow} 0 \quad\text{and}\quad \lim_{m\to -\infty} e^{m\theta_1}X_{m,t_2} \overset{\p}{\longrightarrow} 0$$
for every $t_1$ and $t_2$.
\item There exists $G=(G_t)_{t\in\mathbb{Z}^2}\in\mathcal{G}_{{\Theta}}$ such that
\begin{equation}
\label{eq:recursion3}
X_t = \langle\hat{\Theta}, \hat{X}_t^-\rangle + \Delta_t G\quad\text{for every } t\in\mathbb{Z}^2,
\end{equation}
where $ \langle\hat{\Theta}, \hat{X}_t^-\rangle$ is given by \eqref{eq:innerproduct-2d}.
\end{enumerate}
The stationary increment field $G\in\mathcal{G}_{{\Theta}}$ in \eqref{eq:recursion3} is unique.
\end{kor}
The proof of Theorem \ref{theo:giff} is split into a series of lemmas and auxiliary theorems. We begin with the following result that is a version of Lamperti theorem. 
\begin{theorem}
\label{thm:lamperti}
If $X = (X_t)_{t\in\mathbb{Z}^N}$ is stationary, then $(\mathcal{L}_\Theta X)_{e^t}$
is $\Theta$-self-similar. Conversely, if $Y = (Y_{e^t})_{t\in\mathbb{Z}^N}$ is $\Theta$-self-similar, then $(\mathcal{L}^{-1}_\Theta Y)_t$ is stationary.
\end{theorem}
\begin{proof}
First, assume that $X$ is stationary. Set $Y_{e^t} = (\mathcal{L}_\Theta X)_{e^t}$ and let $n \in \mathbb{N}$. Now
\begin{equation*}
\begin{split}
(Y_{e^{t_1+s}}, \dots, Y_{e^{t_n+s}}) &= (e^{\langle t_1+s,\Theta\rangle} X_{t_1+s}, \dots e^{\langle t_n+s,\Theta\rangle}X_{t_n+s})\\
&\overset{\text{law}}{=} (e^{\langle s,\Theta\rangle} e^{\langle t_1, \Theta\rangle} X_{t_1}, \dots, e^{\langle s,\Theta\rangle} e^{\langle t_n,\Theta\rangle} X_{t_n})\\
&= (e^{\langle s,\Theta\rangle} Y_{e^{t_1}}, \dots, e^{\langle s,\Theta\rangle} Y_{e^{t_n}} ),
\end{split}
\end{equation*}
proving the first part of the claim. Next, assume that $Y$ is $\Theta$-self-similar. Set $X_t = (\mathcal{L}^{-1}_\Theta Y)_t$ and let $n\in\mathbb{N}$. Now
\begin{equation*}
\begin{split}
(X_{t_1+s}, \dots, X_{t_n+s}) &= ( e^{-\langle t_1+s,\Theta\rangle} Y_{e^{t_1+s}}, \dots, e^{-\langle t_n+s,\Theta\rangle} Y_{e^{t_n+s}})\\
& \overset{\text{law}}{=} ( e^{-\langle t_1,\Theta\rangle}Y_{e^{t_1}}, \dots, e^{-\langle t_n,\Theta\rangle} Y_{e^{t_n}})\\
&= (X_{t_1}, \dots, X_{t_n}),
\end{split}
\end{equation*}
completing the proof.
\end{proof}
The following lemma provides one of our key observations.
\begin{lemma}
\label{lemma:gG}
Let $(Y_{e^t})_{t\in\mathbb{Z}^N}$ be $\Theta$-self-similar. Set 
\begin{equation*}
\Delta_t Y = \sum_{(i_1,\dots,i_N)\in\{0,1\}^N} (-1)^{\sum_{l=1}^N i_l} Y_{e^{t_1-i_1},\dots,e^{t_N-i_N}}.
\end{equation*}
For $\sum_{l=1}^N t_l \geq 1$, we set 
$$
G_t = \sum_{k_1= 1-t_2-\dots-t_N}^{t_1} \sum_{k_2=1-k_1-t_3-\dots-t_N}^{t_2} \dots \sum_{k_N = 1-k_1-\dots-k_{N-1}}^{t_N} e^{-\langle k,\Theta\rangle} \Delta_k Y,
$$
and, for $\sum_{l=1}^N t_l \leq 0$, we set 
$$
G_t = (-1)^N\sum_{k_1=t_1+1}^{-t_2-\dots-t_N-N+1} \sum_{k_2=t_2+1}^{-k_1-t_3-\dots-t_N-N+2}\dots \sum_{k_N = t_N+1}^{-k_1-\dots-k_{N-1}} e^{-\langle k, \Theta\rangle} \Delta_k Y.
$$
Here $\langle k,\Theta\rangle$ is the standard inner product and sums of the type $\sum_{s_2}^{s_1}$ with $s_1<s_2$ are interpreted as empty sums. Now 
\begin{enumerate}[label=(\roman*)]
\item $G_t = 0$ for all $t$ such that $\sum_{l=1}^N t_l \in\{0,-1,\dots,-N+1\}$,
\item $\Delta_t G = e^{-\langle t,\Theta\rangle} \Delta_t Y$ for every $t\in\mathbb{Z}^N$,
\item $G = (G_t)_{t\in\mathbb{Z}^N}$ is a stationary increment field.
\end{enumerate}
\end{lemma}
\begin{rem}
It turns out that $G$ defined as above satisfies $G \in \mathcal{G}_{\Theta}$, see also Lemma \ref{lemma:GinG} below. 
\end{rem}
\begin{ex}
\label{lemma:G}
In the two-dimensional case, for $\Theta$-self-similar $(Y_{e^t})_{t\in\mathbb{Z}^2}$, we denote 
$$\Delta_t Y = Y_{e^{t_1}, e^{t_2}}-Y_{e^{t_1-1}, e^{t_2}}-Y_{e^{t_1}, e^{t_2-1}}+Y_{e^{t_1-1}, e^{t_2-1}}.$$
The field $G = (G_t)_{t\in\mathbb{Z}^2}$ defined as
\begin{equation*}
G_{t_1,t_2} = \begin{cases}
\sum_{k_1= 1-t_2}^{t_1} \sum_{k_2=1-k_1}^{t_2} e^{- \langle(k_1, k_2), \Theta\rangle} \Delta_k Y,& \quad t_1+t_2 \geq 1\\
\sum_{k_1=t_1+1}^{-t_2-1} \sum_{k_2=t_2+1}^{-k_1} e^{-\langle( k_1, k_2), \Theta\rangle} \Delta_k Y,& \quad t_1+t_2 \leq0,
\end{cases}
\end{equation*}
belongs to the class $\mathcal{G}_\Theta$. Here sums of the type $\sum_{s+1}^s$ are interpreted as empty sums. 
\end{ex}
The proof of Lemma \ref{lemma:gG} is based on the following additional lemmas that we prove first. The first one provides an auxiliary result on sums of binomial coefficients. Although the result is quite elementary, we provide a proof for the reader's convenience.
\begin{lemma}
\label{lemma:auxiliary}
We have the following identities:
\begin{equation*}
\sum_{m=0}^{\frac{M-1}{2}} \binom{M}{2m} = 2^{M-1}, \quad\sum_{m=0}^{\frac{M-1}{2}} \binom{M}{2m+1} = 2^{M-1}, \qquad\text{when $M\geq 1$ is odd}.
\end{equation*}
\begin{equation*}
\sum_{m=0}^{\frac{M}{2}} \binom{M}{2m} = 2^{M-1}, \quad\sum_{m=0}^{\frac{M}{2}-1} \binom{M}{2m+1} = 2^{M-1}, \qquad\text{when $M\geq 2$ is even}.
\end{equation*}
\end{lemma}
\begin{proof}
In the odd case
\begin{equation*}
\sum_{m=0}^{\frac{M-1}{2}} \binom{M}{2m} = \binom{M}{0}+\ldots+\binom{M}{M-1}\quad\text{and}\quad \sum_{m=0}^{\frac{M-1}{2}} \binom{M}{2m+1} = \binom{M}{1}+\ldots+\binom{M}{M}.
\end{equation*}
These sums are equal since $\binom{M}{k} = \binom{M}{M-k}$. In addition,
\begin{equation*}
\label{binomialsum}
\sum_{m=0}^{\frac{M-1}{2}} \binom{M}{2m}+\sum_{m=0}^{\frac{M-1}{2}} \binom{M}{2m+1} = 2^M
\end{equation*}
completing the proof of the first case. For the even case, we obtain
\begin{equation*}
\begin{split}
\sum_{m=0}^{\frac{M}{2}-1} \binom{M}{2m+1}&= \binom{M}{1}+\ldots+\binom{M}{M-1}\\
 &= \binom{M-1}{0} + \binom{M-1}{1}+\dots+\binom{M-1}{M-2}+\binom{M-1}{M-1} = 2^{M-1}.
\end{split}
\end{equation*}
Observing that 
\begin{equation*}
\sum_{m=0}^{\frac{M}{2}} \binom{M}{2m}  + \sum_{m=0}^{\frac{M}{2}-1} \binom{M}{2m+1} = 2^M
\end{equation*}
completes the proof.
\end{proof}
The following lemma sheds light on how the field $G$ of Lemma \ref{lemma:gG} is constructed from a self-similar field $Y$.
\begin{lemma}
\label{lemma:auxiliary2}
\begin{enumerate}[label=(\roman*)]
\item Let $\sum_{l=1}^N t_l \geq 1$. Then a term $e^{-\langle j,\Theta\rangle} \Delta_j Y$ belongs to the sum defining $G_t$ in Lemma \ref{lemma:gG} if and only if
$$j_l \leq t_l\quad\text{for every $l$ and}\quad \sum_{l=1}^N j_l \geq 1.$$
\item Let $\sum_{l=1}^N t_l \leq -N$. Then a term $e^{-\langle j,\Theta\rangle} \Delta_j Y$ belongs to the sum defining $G_t$ in Lemma \ref{lemma:gG} if and only if
$$j_l \geq t_l+1 \quad\text{for every $l$ and}\quad \sum_{l=1}^N j_l \leq 0.$$
\end{enumerate}
\end{lemma}
\begin{proof}
Item $(i)$: By the upper bounds in the sum defining $G_t$, it is clear that we have $j_l \leq t_l$ for all $l$. By the lower bound of the inmost summation, we obtain 
$j_N \geq 1- j_1-\dots-j_{N-1}$. That is, $\sum_{l=1}^N j_l \geq 1.$
We also observe that the other lower bounds of the sum defining $G_t$ yield conditions 
$$j_{N-h} \geq 1- \sum_{l=1}^{N-h-1} j_l - \sum_{l=N-h+1}^N t_l\quad \text{for every }h\in\{0, \dots, N-1\}.$$
These conditions are satisfied since
\begin{equation*}
\sum_{l=1}^{N-h} j_l + \sum_{N-h+1}^N t_l \geq \sum_{l=1}^{N} j_l \geq 1.
\end{equation*}
This completes the proof of the first item.\\
\newline
Item $(ii)$: By the lower bounds in the sum defining $G_t$, it is clear that we have $j_l \geq t_l+1$ for all $l$. By the upper bound of the inmost summation, we obtain also that
$
j_N \leq - j_1-\dots-j_{N-1}$. That is, $\sum_{l=1}^N j_l \leq 0.$
We also observe that the other upper bounds of the sum defining $G_t$ yield conditions 
$$j_{N-h} \leq - \sum_{l=1}^{N-h-1} j_l - \sum_{l=N-h+1}^N t_l - N + (N-h)\quad \text{for every }h\in\{0, \dots, N-1\}.$$
These conditions are satisfied since
\begin{equation*}
\begin{split}
\sum_{l=1}^{N-h} j_l + \sum_{N-h+1}^N t_l &\leq \sum_{l=1}^{N-h} j_l  + \sum_{N-h+1}^N (j_l -1) =  \sum_{l=1}^{N} j_l - h \leq -h.
\end{split}
\end{equation*}
This completes the proof of the second item, and thus the whole proof is completed.
\end{proof}
\begin{proof}[Proof of Lemma \ref{lemma:gG}]
Item $(i)$: Let $\sum_{l=1}^N t_l \in\{0,-1,\dots,-N+1\}$. Then 
$$\sum_{l=2}^N t_l \in\{-t_1,-1-t_1,\dots,-N+1-t_1\}$$
and for the upper bound of the first summation in the definition of $G_t$ it holds that
$$-\sum_{l=2}^N t_l - N +1 \in\{t_1-N+1,t_1-N+2,\dots,t_1\}.$$
Hence, $G_t$ is given by an empty sum.\\
\newline
Item $(ii)$: Recall that  $G_t = 0$ for all $t$ such that $\sum_{i=1}^N t_i \in\{0,-1,\dots,-N+1\}$. 

We treat the case $\sum_{l=1}^N t_l \geq 1$ first. Let $M$ be such that $\sum_{l=1}^N t_l -M =1$. Then
\begin{equation}
\label{deltaG}
\Delta_t G = \sum_{\substack{(i_1,\dots,i_N)\in\{0,1\}^N\\ \sum_{l=1}^N i_l \leq M}} (-1)^{\sum_{l=1}^N i_l} G_{t_1-i_1,\dots,t_N-i_N}.
\end{equation}
By Lemma \ref{lemma:auxiliary2}, $\Delta_t G$ consists of terms $e^{-\langle j,\Theta\rangle} \Delta_j Y$ with $j_l \leq t_l$ for every $l$ and $\sum_{l=1}^N j_l \geq 1$. 
Let $m$ be the number of indices $l$ for which $j_l = t_l$. \\
Assume that $m < N$.
By Lemma \ref{lemma:auxiliary2}, $e^{-\langle j,\Theta\rangle} \Delta_j Y$ belongs to summands of \eqref{deltaG} that satisfy $j_l \leq t_l -i_l$ for every $l$. That is, $m$ of the indices $i_l$ are zero while the remaining $N-m$ indices may be zeros or ones. 
In addition, 
$$ 1\leq \sum_{l=1}^N j_l \leq \sum_{l=1}^N t_l - (N-m)$$
giving
\begin{equation}
\label{bound}
N-m \leq \sum_{l=1}^N t_l - 1 = M.
\end{equation}
Now if $N-m$ is odd, then by Lemma \ref{lemma:auxiliary} and \eqref{bound}, the number of terms $e^{-\langle j,\Theta\rangle} \Delta_j Y$ in \eqref{deltaG} with a positive sign is 
\begin{align*}
\binom{N-m}{0}+\ldots+\binom{N-m}{\min\{N-m, M\} - 1} = \binom{N-m}{0}+\ldots+ \binom{N-m}{N-m-1} = 2^{N-m-1}.
\end{align*}
Thus, terms $e^{-\langle j,\Theta\rangle} \Delta_j Y$ cancel out in \eqref{deltaG}. \\
Similarly, if $N-m$ is even, then the number of terms $e^{-j\Theta} \Delta_j Y$ with a positive sign is 
\begin{equation*}
\binom{N-m}{0}+\dots+\binom{N-m}{\min\{N-m, M\}} = \binom{N-m}{0}+\dots+ \binom{N-m}{N-m}.
\end{equation*}
Again, by Lemma \ref{lemma:auxiliary}, terms $e^{-\langle j,\Theta\rangle} \Delta_j Y$ cancel out in \eqref{deltaG}.\\
If $m=N$, we have that $e^{-\langle j,\Theta\rangle} \Delta_j Y = e^{-\langle t,\Theta\rangle} \Delta_t Y$ belongs only to the summand of \eqref{deltaG} with $i=\bf{0}$. Hence we have shown that
$$ \Delta_t G = e^{-\langle t,\Theta\rangle} \Delta_t Y\quad\text{for every $t$ such that } \sum_{l=1}^N t_l \geq 1.$$
This proves the claim for the case $\sum_{l=1}^N t_l \geq 1$. 

Assume next that $\sum_{l=1}^N t_l \leq 0$ and let $M$ be such that $\sum_{l=1}^N t_l - M = -N$. Then 
\begin{equation}
\label{deltaG2}
\Delta_t G = \sum_{\substack{(i_1,\dots,i_N)\in\{0,1\}^N\\ \sum_{l=1}^N i_l \geq M}} (-1)^{\sum_{l=1}^N i_l} G_{t_1-i_1,\dots,t_N-i_N}.
\end{equation}
By Lemma \ref{lemma:auxiliary2}, $\Delta_t G$ consists of terms $e^{-\langle j,\Theta\rangle} \Delta_j Y$ with $j_l \geq t_l$ for every $l$ and $\sum_{l=1}^N j_l \leq 0$. 
As before, let $m$ be the number of indices $l$ for which $j_l = t_l$. If $m < N$, then,
by Lemma \ref{lemma:auxiliary2}, $e^{-\langle j,\Theta\rangle} \Delta_j Y$ belongs to summands of \eqref{deltaG2} that satisfy $j_l \geq t_l -i_l+1$ for every $l$. That is, $m$ of the indices $i_l$ are equal to one while the remaining $N-m$ indices may be zeros or ones. 
In addition,
$$ 0 \geq \sum_{l=1}^N j_l \geq \sum_{l=1}^N t_l + (N-m) = M-m$$
giving $m \geq M$ and $N-m\leq N-M$.
If $N-m$ is odd, by Lemma \ref{lemma:auxiliary}, the number of terms $e^{-\langle j,\Theta\rangle} \Delta_j Y$ in \eqref{deltaG2} with the sign $(-1)^{N+m}$ is equal to
\begin{align*}
\binom{N-m}{0}+\dots+\binom{N-m}{N-m-1} = 2^{N-m-1}. 
\end{align*}
Thus, terms cancel out in \eqref{deltaG2}. Similarly, terms cancel out when $N-m$ is even. Finally, for the case $m=N$ we observe that $e^{-\langle j,\Theta\rangle} \Delta_j Y = e^{-\langle t,\Theta\rangle} \Delta_t Y$ belongs only to the summand of \eqref{deltaG2} with $i=\bf{1}$. The corresponding sign is $(-1)^N(-1)^N = 1$. \\
To conclude, we have shown that
$$ \Delta_t G = e^{-\langle t,\Theta\rangle} \Delta_t Y\quad\text{for every $t$ such that } \sum_{l=1}^N t_l \leq 0.$$
This completes the proof of item $(ii)$.
 \\
\newline
Item $(iii)$: For $s= (s_1,\dots,s_N)\in\mathbb{Z}^N$,
\begin{equation*}
\begin{split}
\Delta_{t+s} G & = e^{-\langle t+s,\Theta\rangle} \Delta_{t+s} Y\\
 &= e^{-\langle t,\Theta\rangle} e^{-\langle s,\Theta\rangle}  \sum_{(i_1,\dots,i_N)\in\{0,1\}^N} (-1)^{\sum_{l=1}^N i_l} Y_{e^{t_1+s_1-i_1},\dots,e^{t_N+s_N-i_N}}\\
 &\overset{\text{law}}{=} e^{-\langle t,\Theta\rangle}\sum_{(i_1,\dots,i_N)\in\{0,1\}^N} (-1)^{\sum_{l=1}^N i_l} Y_{e^{t_1-i_1},\dots,e^{t_N-i_N}}\\
 &= e^{-\langle t,\Theta\rangle} \Delta_t Y = \Delta_t G.
\end{split}
\end{equation*}
Treating multidimensional distribution similarly completes the proof of item $(iii)$.
\end{proof}
We are now ready to prove three results, Theorem \ref{theo:gchar1}, Theorem \ref{theo:gchar2}, and Lemma \ref{lemma:guniqueness}, that give us the main result of this article, Theorem \ref{theo:giff}.
\begin{theorem}
\label{theo:gchar1}
Let $\Theta = (\theta_1, \dots, \theta_N)\in (0,\infty)^N$ and let $X = (X_t)_{t\in\mathbb{Z}^N}$ be a random field. If for some $G=(G_t)_{t\in\mathbb{Z}^N}\in\mathcal{G}_{{\Theta}}$ it holds that
\begin{equation}
\label{eq:grecursion2}
X_t = \langle\hat{\Theta}, \hat{X}_t^-\rangle + \Delta_t G\quad\text{for every } t\in\mathbb{Z}^N,
\end{equation}
and
$$\lim_{m\to -\infty} e^{m\theta_j}X_{t_1,\dots,t_{j-1},m,t_{j+1},\dots,t_N} \overset{\p}{\longrightarrow} 0$$
for every $j\in\{1,\dots,N\}$ and $t_1, \dots,t_{j-1},t_{j+1},\dots, t_N \in \mathbb{Z}$, then $X=(X_t)_{t\in\mathbb{Z}^N}$ is stationary.
\end{theorem}
\begin{proof}
Denote $Z_t = \Delta_t G$ and $i=(i_1,\ldots,i_N)$. From \eqref{eq:grecursion2} and Definition \ref{defi:inner2}, we get
\begin{equation*}
X_t = \sum_{\substack{(i_1,\dots,i_N)\in\{0,1\}^N\\ i\neq \bf{0}}} (-1)^{1+\sum_{l=1}^N i_l} e^{-\langle i,\Theta\rangle} X_{{t_1-i_1},\dots,{t_N-i_N}} + Z_t,
\end{equation*}
which gives
\begin{equation}
\label{firststep}
\begin{split}
&X_t - \sum_{\substack{(i_1,\dots,i_{N-1})\in\{0,1\}^{N-1}\\ i\neq \bf{0}}} (-1)^{1+\sum_{l=1}^{N-1} i_l} e^{ -\langle(i_1, \dots,i_{N-1}, 0 ), \Theta\rangle} X_{{t_1-i_1},\dots, t_{N-1}-i_{N-1}, {t_N}}\\
&= \sum_{(i_1,\dots,i_{N-1})\in\{0,1\}^{N-1}} (-1)^{\sum_{l=1}^{N-1} i_l} e^{ -\langle(i_1, \dots,i_{N-1}, 0 ), \Theta\rangle} X_{{t_1-i_1},\dots, t_{N-1}-i_{N-1}, {t_N}}\\
&= \sum_{(i_1,\dots,i_{N-1})\in\{0,1\}^{N-1}} (-1)^{\sum_{l=1}^{N-1} i_l} e^{ -\langle(i_1, \dots,i_{N-1}, 1 ),\Theta\rangle} X_{{t_1-i_1},\dots, t_{N-1}-i_{N-1}, {t_N-1}} + Z_t.
\end{split}
\end{equation}
Set
\begin{equation*}
Y_{t_1,\dots,t_{N-1}}(t_N) = \sum_{(i_1,\dots,i_{N-1})\in\{0,1\}^{N-1}} (-1)^{\sum_{l=1}^{N-1} i_l} e^{ -\langle(i_1, \dots,i_{N-1}, 0 ), \Theta\rangle} X_{{t_1-i_1},\dots, t_{N-1}-i_{N-1}, {t_N}}.
\end{equation*}
Then, by iterating the recursive Equation \eqref{firststep}, we get
\begin{equation*}
\begin{split}
Y_{t_1,\dots,t_{N-1}}(t_N) 
=& e^{-\theta_N} Y_{t_1,\dots,t_{N-1}}(t_N-1) + Z_t \\
=& e^{-(n+1)\theta_N} Y_{t_1,\dots,t_{N-1}}(t_N-n-1) + \sum_{j_N=0}^n e^{-j_N\theta_N} Z_{t_1,\dots,t_{N-1}, t_N-j_N}\\
=& e^{-(n+1)\theta_N} Y_{t_1,\dots,t_{N-1}}(t_N-n-1) + e^{-t_N\theta_N} \sum_{j_N=t_N -n}^{t_N} e^{j_N\theta_N}Z_{t_1,\dots,t_{N-1},j_N},
\end{split}
\end{equation*}
for every $n\in\mathbb{N}$. Above
\begin{equation*}
\begin{split}
e^{-(n+1)\theta_N} Y_{t_1,\dots,t_{N-1}}(t_N-n-1) &= e^{-t_N\theta_N} e^{(t_N-n-1)\theta_N}Y_{t_1,\dots,t_{N-1}}(t_N-n-1)\\
&= e^{-t_N\theta_N} e^{m\theta_N}Y_{t_1,\dots,t_{N-1}}(m)
\end{split}
\end{equation*}
by the change  of variable $m=t_N-n-1$. Furthermore
\begin{equation*}
\begin{split}
&e^{m\theta_N}Y_{t_1,\dots,t_{N-1}}(m) \\
=& e^{m\theta_N} \sum_{(i_1,\dots,i_{N-1})\in\{0,1\}^{N-1}} (-1)^{\sum_{l=1}^{N-1} i_l} e^{ -\langle(i_1, \dots,i_{N-1}, 0 ), \Theta\rangle} X_{{t_1-i_1},\dots, t_{N-1}-i_{N-1}, m},
\end{split}
\end{equation*}
which, by the assumptions, converges to zero in probability as $m\to -\infty$.
Hence
\begin{equation*}
\begin{split}
& Y_{t_1,\dots,t_{N-1}}(t_N) \\
&=  \sum_{(i_1,\dots,i_{N-1})\in\{0,1\}^{N-1}} (-1)^{\sum_{l=1}^{N-1} i_l} e^{ -\langle(i_1, \dots,i_{N-1}, 0 ), \Theta\rangle} X_{{t_1-i_1},\dots, t_{N-1}-i_{N-1}, {t_N}}\\
&= e^{-t_N\theta_N} \sum_{j_N=-\infty}^{t_N} e^{j_N\theta_N}Z_{t_1,\dots,t_{N-1},j_N} \eqqcolon Q_t^{(N)}.
\end{split}
\end{equation*}
In the following summations, let $(i_1,\dots,i_{N-k-1}) \in \{0,1\}^{N-k-1}$  and $(i_1,\dots,i_{N-k-2})\in\{0,1\}^{N-k-2}$. 
We proceed by induction and assume that for some $k \in \mathbb{N}\cup \{0\}$ it holds that
\begin{align}
&\sum_{(i_1,\dots,i_{N-k-1})} (-1)^{\sum_{l=1}^{N-k-1} i_l} e^{ -\langle(i_1, \dots,i_{N-k-1}, 0, \dots,0 ), \Theta\rangle} X_{{t_1-i_1},\dots, t_{N-k-1} - i_{N-k-1}, t_{N-k},\dots, {t_N}}\nonumber\\
=& e^{-\sum_{l=0}^k t_{N-l}\theta_{N-l}} \sum_{j_{N-k}=-\infty}^{t_{N-k}} \dots \sum_{j_N=-\infty}^{t_N} e^{\sum_{l=0}^k j_{N-l}\theta_{N-l}} Z_{t_1,\dots,t_{N-k-1},j_{N-k},\dots,j_N} 
\eqqcolon  Q_t^{(N-k)} \label{induction}.
\end{align}
Now
\begin{align}
&\sum_{(i_1,\dots,i_{N-k-2})} \big((-1)^{\sum_{l=1}^{N-k-2} i_l} e^{ -\langle(i_1, \dots,i_{N-k-2}, 0, \dots,0 ), \Theta\rangle} \nonumber\\
&\cdot X_{{t_1-i_1},\dots, t_{N-k-2} - i_{N-k-2}, t_{N-k-1},\dots, {t_N}}\big)\nonumber\\
=&-\sum_{(i_1,\dots,i_{N-k-2})} \big((-1)^{1+\sum_{l=1}^{N-k-2} i_l} e^{ -\langle(i_1, \dots,i_{N-k-2}, 1,0, \dots,0 ), \Theta\rangle} \nonumber\\
&\cdot X_{{t_1-i_1},\dots, t_{N-k-2} - i_{N-k-2}, t_{N-k-1}-1,\dots, {t_N}}\big)+Q_t^{(N-k)}\nonumber\\
=&\sum_{(i_1,\dots,i_{N-k-2})} \big((-1)^{\sum_{l=1}^{N-k-2} i_l} e^{ -\langle(i_1, \dots,i_{N-k-2}, 1,0, \dots,0 ), \Theta\rangle}\nonumber \\
&\cdot X_{{t_1-i_1},\dots, t_{N-k-2} - i_{N-k-2}, t_{N-k-1}-1,\dots, {t_N}}\big)+Q_t^{(N-k)}. \label{inductionstep}
\end{align}
Let $t^* = (t_1,\dots, t_{N-k-2},t_{N-k},\dots, t_N)$ and define
\begin{equation*}
\begin{split}
&Y_{t^*} (t_{N-k-1})\\
=&\sum_{(i_1,\dots,i_{N-k-2})} (-1)^{\sum_{l=1}^{N-k-2} i_l} e^{ -\langle(i_1, \dots,i_{N-k-2}, 0, \dots,0 ), \Theta\rangle} X_{{t_1-i_1},\dots, t_{N-k-2} - i_{N-k-2}, t_{N-k-1},\dots, {t_N}}.
\end{split}
\end{equation*}
Equation \eqref{inductionstep} gives
\begin{equation*}
\begin{split}
Y_{t^*} (t_{N-k-1}) 
=& e^{-\theta_{N-k-1}}Y_{t^*} (t_{N-k-1}-1) + Q_t^{(N-k)} \\
=& e^{-(n+1)\theta_{N-k-1}} Y_{t^*} (t_{N-k-1}-n-1) \\
&+ \sum_{j_{N-k-1}=0}^n e^{-j_{N-k-1}\theta_{N-k-1}} Q^{(N-k)}_{t_1,\dots,t_{N-k-2},t_{N-k-1}-j_{N-k-1}, t_{N-k}, \dots, t_N} \\
=&  e^{-(n+1)\theta_{N-k-1}} Y_{t^*} (t_{N-k-1}-n-1) \\
&+e^{-t_{N-k-1}\theta_{N-k-1}}\sum_{\mathclap{j_{N-k-1}=t_{N-k-1}-n}}^{t_{N-k-1}} e^{j_{N-k-1}\theta_{N-k-1}}Q^{(N-k)}_{t_1,\dots,t_{N-k-2},j_{N-k-1}, t_{N-k}, \dots, t_N}
\end{split}
\end{equation*}
for every $n\in\mathbb{N}$. Above
\begin{equation*}
\begin{split}
e^{-(n+1)\theta_{N-k-1}} Y_{t^*} (t_{N-k-1}-n-1) 
=& e^{-t_{N-k-1}\theta_{N-k-1}} e^{(t_{N-k-1}-n-1)\theta_{N-k-1}}Y_{t^*} (t_{N-k-1}-n-1)\\
=&  e^{-t_{N-k-1}\theta_{N-k-1}} e^{m\theta_{N-k-1}}Y_{t^*}(m)
\end{split}
\end{equation*}
by the change of variable $m=t_{N-k-1}-n-1$. As before, the expression converges to zero in probability as $m\to\-\infty$. Hence, we obtain that
\begin{equation*}
\begin{split}
&Y_{t^*} (t_{N-k-1})\\
=&\sum_{(i_1,\dots,i_{N-k-2})} \big((-1)^{\sum_{l=1}^{N-k-2} i_l} e^{ -\langle(i_1, \dots,i_{N-k-2}, 0, \dots,0 ), \Theta\rangle} \\
&\cdot X_{{t_1-i_1},\dots, t_{N-k-2} - i_{N-k-2}, t_{N-k-1},\dots, {t_N}}\big)\\
=& e^{-t_{N-k-1}\theta_{N-k-1}}\sum_{\mathclap{j_{N-k-1}=-\infty}}^{t_{N-k-1}} e^{j_{N-k-1}\theta_{N-k-1}}Q^{(N-k)}_{t_1,\dots,t_{N-k-2},j_{N-k-1}, t_{N-k}, \dots, t_N}\\
=& e^{-t_{N-k-1}\theta_{N-k-1}}\Big(\sum_{\mathclap{j_{N-k-1}=-\infty}}^{t_{N-k-1}} e^{j_{N-k-1}\theta_{N-k-1}}e^{-\sum_{l=0}^k t_{N-l}\theta_{N-l}} \\
&\cdot \sum_{j_{N-k}=-\infty}^{t_{N-k}} \dots \sum_{j_N=-\infty}^{t_N} e^{\sum_{l=0}^k j_{N-l}\theta_{N-l}} Z_{t_1,\dots,t_{N-k-2},j_{N-k-1},\dots,j_N}\Big)\\
 =& e^{-\sum_{l=0}^{k+1} t_{N-l}\theta_{N-l}} \sum_{j_{N-k-1}=-\infty}^{t_{N-k-1}} \dots \sum_{j_N=-\infty}^{t_N} e^{\sum_{l=0}^{k+1} j_{N-l}\theta_{N-l}} Z_{t_1,\dots,t_{N-k-2},j_{N-k-1},\dots,j_N},
\end{split}
\end{equation*}
which proves the induction step. Therefore choosing $k= N-2$ in \eqref{induction} yields
\begin{equation*}
\begin{split}
&\sum_{i_1\in\{0,1\}} (-1)^{i_1} e^{-\langle(i_1,0,\dots,0),\Theta\rangle} X_{t_1-i_1,t_2,\dots,t_N} = X_t - e^{-\theta_1}X_{t_1-1,t_2,\dots,t_N}\\
&=  e^{-\sum_{l=0}^{N-2} t_{N-l}\theta_{N-l}} \sum_{j_2=-\infty}^{t_2} \dots \sum_{j_N=-\infty}^{t_N} e^{\sum_{l=0}^{N-2} j_{N-l}\theta_{N-l}} Z_{t_1,j_2,\dots,j_N} \eqqcolon Q_t^{(2)}
\end{split}
\end{equation*}
and we obtain a recursive equation
\begin{equation*}
X_t = e^{-\theta_1}X_{t_1-1,t_1,\dots,t_N} + Q_t^{(2)}.
\end{equation*}
By repeating the earlier procedure once more, we obtain that 
\begin{equation*}
\begin{split}
X_t 
=& e^{-t_1\theta_1} \sum_{j_1=-\infty}^{t_1} e^{j_1\theta_1} Q^{(2)}_{j_1,t_2,\dots,t_N}\\
=& e^{-t_1\theta_1} \sum_{j_1=-\infty}^{t_1} e^{j_1\theta_1}e^{-\sum_{l=0}^{N-2} t_{N-l}\theta_{N-l}} \sum_{j_2=-\infty}^{t_2} \dots \sum_{j_N=-\infty}^{t_N} e^{\sum_{l=0}^{N-2} j_{N-l}\theta_{N-l}} Z_{j_1,j_2,\dots,j_N}\\
=& e^{-\sum_{l=1}^N t_l\theta_l}\sum_{j_1=-\infty}^{t_1} \dots \sum_{j_N=-\infty}^{t_N} e^{\sum_{l=1}^N j_l\theta_l} Z_{j_1,j_2,\dots,j_N}\\
=&e^{-\sum_{l=1}^N t_l\theta_l}\sum_{j_1=-\infty}^{t_1} \dots \sum_{j_N=-\infty}^{t_N} e^{\sum_{l=1}^N j_l\theta_l} \Delta_{j_1,j_2,\dots,j_N} G,
\end{split}
\end{equation*}
which, since $G\in\mathcal{G}_{{\Theta}}$, defines an almost surely finite random variable. Hence it remains to prove that $X$ is stationary. To this end, we show that the one dimensional distributions of $X$ are stationary. The proof extends straightforwardly to multidimensional distributions. Let $s= (s_1,\dots,s_N) \in\mathbb{Z}^N$. Note that
$$X_t = \sum_{j_1=0}^\infty \dots \sum_{j_N=0}^\infty e^{-\sum_{l=1}^N j_l\theta_l}\Delta_{t_1-j_1,\dots,t_N-j_N}G.$$
Since $G$ is a stationary increment field, we have that
\begin{equation*}
\begin{split}
&\sum_{j_1=0}^{M_1} \dots\sum_{j_N=0}^{M_N} e^{-\sum_{l=1}^N j_l\theta_l}\Delta_{t_1+s_1-j_1,\dots,t_N+s_N-j_N}G \\
\overset{\text{law}}{=}& \sum_{j_1=0}^{M_1} \dots\sum_{j_N=0}^{M_N} e^{-\sum_{l=1}^N j_l\theta_l}\Delta_{t_1-j_1,\dots,t_N-j_N}G
\end{split}
\end{equation*}
for every $M_1, \dots,M_N\in\mathbb{N}$. Moreover, since $G\in\mathcal{G}_{{\Theta}}$, the iterated limits of both sides converge and hence, the limits are equal in distribution. This gives
\begin{equation*}
\begin{split}
X_{t+s} &= \sum_{j_1=0}^{\infty} \dots\sum_{j_N=0}^{\infty} e^{-\sum_{l=1}^N j_l\theta_l}\Delta_{t_1+s_1-j_1,\dots,t_N+s_N-j_N}G\\
 &\overset{\text{law}}{=} \sum_{j_1=0}^{\infty} \dots\sum_{j_N=0}^{\infty} e^{-\sum_{l=1}^N j_l\theta_l}\Delta_{t_1-j_1,\dots,t_N-j_N}G = X_t
\end{split}
\end{equation*}
and thus the proof is completed.
\end{proof}
\begin{rem}
\label{rem:peculiar-condition}
The property $G_t = 0$ for all $t$ such that $\sum_{i=1}^N t_i \in\{0,-1,\dots,-N+1\}$ of the class $\mathcal{G}_{\Theta}$ is not utilized in the proof of Theorem \ref{theo:gchar1}. However, in order to obtain uniqueness in the characterising Equation \eqref{eq:grecursion2}, we need to pose this additional assumption, see Lemma \ref{lemma:guniqueness}.
\end{rem}
\begin{theorem}
\label{theo:gchar2}
Let $\Theta = (\theta_1, \dots, \theta_N)\in (0,\infty)^N$. Assume that the field $X = (X_t)_{t\in\mathbb{Z}^N}$ is stationary. Then there exists $G = (G_t)_{t\in\mathbb{Z}^N}\in\mathcal{G}_{{\Theta}}$ such that
$$X_t = \langle\hat{\Theta}, \hat{X}_t^-\rangle + \Delta_t G\quad\text{for every } t\in\mathbb{Z}^N.$$
\end{theorem}
\begin{proof}
Applying Lamperti transformation gives
\begin{equation*}
\begin{split}
\Delta_t X &= X_t - X_t^- = e^{-\langle t,\Theta\rangle} Y_{e^t} - X^-_t = e^{-\langle t,\Theta\rangle} ( Y_{e^t} - \Delta_t Y) - X_t^- + e^{-\langle t,\Theta\rangle} \Delta_t Y.
\end{split}
\end{equation*}
Hence
\begin{equation*}
X_t = e^{-\langle t,\Theta\rangle} ( Y_{e^t} - \Delta_t Y) + e^{-\langle t,\Theta\rangle} \Delta_t Y,
\end{equation*}
where
\begin{equation*}
 Y_{e^t} - \Delta_t Y = \sum_{(i_1,\dots,i_N)\in\{0,1\}^N, i\neq \bf{0}} (-1)^{1+\sum_{l=1}^N i_l} Y_{e^{t_1-i_1},\dots,e^{t_N-i_N}}.
\end{equation*}
Furthermore
\begin{equation*}
\begin{split}
 e^{-\langle t,\Theta\rangle} ( Y_{e^t} - \Delta_t Y) 
=& \sum_{(i_1,\dots,i_N)\in\{0,1\}^N, i\neq \bf{0}} (-1)^{1+\sum_{l=1}^N i_l}e^{-\langle t-i,\Theta\rangle} e^{-\langle i,\Theta\rangle} Y_{e^{t_1-i_1},\dots,e^{t_N-i_N}}\\
 =& \sum_{(i_1,\dots,i_N)\in\{0,1\}^N, i\neq\bf{0}} (-1)^{1+\sum_{l=1}^N i_l} e^{-\langle i,\Theta\rangle} X_{{t_1-i_1},\dots,{t_N-i_N}} = \langle\hat{\Theta},\hat{X}_t^-\rangle,
\end{split}
\end{equation*}
where the last equality follows by Lamperti transformation. Let $G$ be the field defined in Lemma \ref{lemma:gG}. Then $X_t = \langle\hat{\Theta}, \hat{X}_t^- \rangle + \Delta_t G$. Moreover, the proof of Theorem \ref{theo:gchar1} and the property
$$\lim_{m\to -\infty} e^{m\theta_i}X_{t_1,\dots,t_{i-1},m,t_{i+1},\dots,t_N} \overset{\p}{\longrightarrow} 0,$$
for every $i\in\{1,\dots,N\}$ and $t_1, \dots,t_{i-1},t_{i+1},\dots, t_N \in \mathbb{Z}$, give us the representation
\begin{equation}
\label{eq:grep}
X_t = e^{-\sum_{l=1}^N t_l\theta_l}\sum_{j_1=-\infty}^{t_1} \dots \sum_{j_N=-\infty}^{t_N} e^{\sum_{l=1}^N j_l\theta_l} \Delta_{(j_1,j_2,\dots,j_N)} G.
\end{equation}
Hence, $G\in\mathcal{G}_{{\Theta}}$.
\end{proof}
As a by-product, we observe the following result.
\begin{lemma}
\label{lemma:GinG}
Let $G$ be defined as in Lemma \ref{lemma:gG}. Then $G\in\mathcal{G}_{{\Theta}}$.
\end{lemma}
\begin{proof}
By Lemma \ref{lemma:gG}, it suffices to show that the limit \eqref{eq:glimitcondition} exists. Now since $Y$ is $\Theta$-self-similar, there exists a stationary $X$ such that $Y = \mathcal{L}_\Theta X$. Following the lines of the proof of Theorem \ref{theo:gchar2}, we obtain Representation \eqref{eq:grep}, where $G$ is defined as in Lemma \ref{lemma:gG}. Hence  $G\in\mathcal{G}_{{\Theta}}$ by the proof of Theorem \ref{theo:gchar2}.
\end{proof}
\begin{lemma}
\label{lemma:guniqueness}
Let $\Theta = (\theta_1, \dots, \theta_N)\in (0,\infty)^N$ be fixed. Then the stationary increment field $G\in\mathcal{G}_{{\Theta}}$ in Theorems \ref{theo:gchar1} and \ref{theo:gchar2} is unique.
\end{lemma}
\begin{proof}
Assume that $G, G'\in\mathcal{G}_{{\Theta}}$ satisfy the recursive Equation \eqref{eq:grecursion2} of Theorems \ref{theo:gchar1} and \ref{theo:gchar2}. Then, $\Delta_t G = \Delta_t G'$ for all $t\in\mathbb{Z}^N$. It remains to show that this implies $G_t = G'_t$ for all $t$, which we do by induction. 

Since $G_t = G'_t = 0$ for all $t$ such that $\sum_{l=1}^N t_l \in\{0,-1,\dots,-N+1\}$, we may set an induction assumption that there exists $s\in\mathbb{Z}$ such that $G_t = G'_t$ for all $t$ such that $\sum_{l=1}^N t_l\in\{s,s-1,\dots,s-N+1\}$. In order to complete the proof, it remains to show that $G_t = G'_t$ under conditions $\sum_{l=1}^N t_l = s+1$ and $\sum_{l=1}^N t_l = s-N$. Let $\sum_{l=1}^N t_l = s+1$. We now have
\begin{equation*}
\begin{split}
\Delta_t G &= \sum_{(i_1,\dots,i_N)\in\{0,1\}^N} (-1)^{\sum_{l=1}^N i_l} G_{t_1-i_1,\dots,t_N-i_N}\\
&= G_t + \sum_{\substack{(i_1,\dots,i_N)\in\{0,1\}^N\\ i\neq \bf{0}}} (-1)^{\sum_{l=1}^N i_l} G_{t_1-i_1,\dots,t_N-i_N},
\end{split}
\end{equation*}
where for the indices of $G$ it holds that
$$\sum_{l=1}^N (t_l-i_l) \in \{s,s-1,\dots,s-N+1\}.$$
Therefore, by the induction assumption, $G_t = G_t'$ for all $t$ such that $\sum_{l=1}^N t_l = s+1$.
We proceed proving that $G_t=G_t'$ for all $t$ such that $\sum_{l=1}^N t_l= s-N$. Let $\sum_{l=1}^N t_l= s-N$. Now 
\begin{equation*}
\begin{split}
 &(-1)^N G_{t_1,\dots,t_N} + \sum_{\substack{(i_1,\dots,i_N)\in\{0,1\}^N\\ i\neq \bf{1}}} (-1)^{\sum_{l=1}^N i_l} G_{t_1+1-i_1,\dots,t_N+1-i_N} \\
 &=(-1)^N G'_{t_1,\dots,t_N} + \sum_{\substack{(i_1,\dots,i_N)\in\{0,1\}^N\\ i\neq \bf{1}}} (-1)^{\sum_{l=1}^N i_l} G'_{t_1+1-i_1,\dots,t_N+1-i_N},
\end{split}
\end{equation*}
where 
$$\sum_{l=1}^N (t_l+1-i_l) \in \{s,s-1,\dots,s-N+1\}.$$
Hence, by the induction assumption, $G_t = G_t'$ for all $t$ such that $\sum_{l=1}^N t_l= s-N$. This completes the proof.
\end{proof}
The proof of our main theorem, Theorem \ref{theo:giff}, now follows by combining Theorems \ref{theo:gchar1} and \ref{theo:gchar2} together with Lemma \ref{lemma:guniqueness}.

\subsection{Fractional Ornstein-Uhlenbeck fields on $\mathbb{Z}^N$}
\label{subsec:example}
We illustrate our approach by constructing a stationary \emph{fractional Ornstein-Uhlenbeck field} on $\mathbb{Z}^2$. For this purpose, let $G_{t_1,t_2}$ be a discrete fractional Gaussian field on $(t_1,t_2) \in \mathbb{N}^2$. A discrete fractional Gaussian field can be constructed from its continuous time analogue by embedding to the space $\mathbb{N}^2$. That is, let $\tilde{X}$ be a self-similar Gaussian field having stationary rectangular increments such that $\mathbb{E} \tilde{X}^2(1,1) = 1$. Existence of such fields is studied in \cite{makogin2019gaussian}. We can now embed $\tilde{X}$ into $\mathbb{N}^2$ by considering values of $\tilde{X}$ at points $(t_1,t_2) \in \mathbb{N}^2$. Now Definition \ref{defi:class} of $\mathcal{G}_{\Theta}$ allows us to extend $G$ into $\mathbb{Z}^2$ by setting $G_{t,-t} = G_{t,-t-1} = 0$ and requiring that $G$ has stationary increments in the sense of Definition \ref{def:stat-inc}. As $G$ is Gaussian, in view of Remark \ref{rem:integrability}, it follows that our field $G=(G_t)_{t\in \mathbb{Z}^2}$ belongs to $\mathcal{G}_{\Theta}$ for any $\Theta$. This allows us to define a discrete time stationary \emph{fractional Ornstein-Uhlenbeck field of the first kind} by Representation \eqref{eq:grecursion3}. On the other hand, using self-similarity of $\tilde{X}$, we can define a stationary field through Lamperti transformation and obtain, in view of Theorem \ref{thm:lamperti}, stationary \emph{fractional Ornstein-Uhlenbeck field of the second kind} corresponding to a different noise $G \in \mathcal{G}_{\Theta}$. Obviously, this approach can be used to define fractional Ornstein-Uhlenbeck fields in arbitrary dimensions $N\geq 2$.
\section{Conclusions}
\label{sec:conclusion}
In this article we have provided a characterisation of discrete stationary fields in terms of fields having stationary rectangular increments. Our characterisation is analogous to the one-parameter case provided in \cite{voutilainen2017model}. As an example, we have extended the notion of fractional Ornstein-Uhlenbeck processes, introduced in \cite{Cheridito-et-al2003} and \cite{Kaarakka-Salminen2011}, to the multi-parameter case.

Natural further prospects of our study are two folded. Firstly, it would be interesting to generalise the characterisation into the continuous parameter space $t \in \mathbb{R}^N$. In this case however, the increments $\Delta_t G$ are replaced with differentials, and the notion of differential $d_t G$ is more subtle when $t\in \mathbb{R}^N$. This is a topic of a forthcoming article. Secondly, parameter estimation is of paramount importance in statistics, which in our setting would correspond to the estimation of the parameter $\Theta$. This has been a topic of active research in the context of fractional Ornstein-Uhlenbeck processes, see e.g. \cite{Hu-Nualart2010,Sottinen-Viitasaari2018} and the references therein. For a more general setting in the context of multivariate time series, parameter estimation is studied in \cite{voutilainen2021vector}. We believe that ideas and methods presented in \cite{voutilainen2021vector} could be useful in our context as well.

\bibliographystyle{plain}
\bibliography{biblio}
\end{document}